\documentclass[]{amsart}
 \usepackage[sorted-cites]{amsrefs}
\usepackage{amsfonts}
\usepackage{graphicx}
\usepackage{amscd}
\usepackage{amsmath}
\usepackage{amssymb}
\usepackage{xcolor}

\makeatletter
\@namedef{subjclassname@2010}{%
  \textup{2010} Mathematics Subject Classification}
\makeatother

\usepackage{mathrsfs}
\usepackage{amsbsy}

\newtheorem{corollary}{\bf Corollary}

\newtheorem{lemma}{\bf Lemma}

\newtheorem{proposition}{\bf Proposition}
\newtheorem{remark}{Remark}

\newtheorem{theorem}{\bf Theorem}

\renewcommand{\div}{{\rm div}}

\def\RR{{\mathrm R}}
\def \2R{{\hat{\RR}}}
\def\WW{{\mathrm W}}

\def\Rc{Ric}

\def\det{\mathrm{det}}

\theoremstyle{definition}

\numberwithin{equation}{section}

\makeatletter
\@namedef{subjclassname@2020}{%
  \textup{2020} Mathematics Subject Classification}
\makeatother

\title[Four-dimensional K\"ahler-Ricci  solitons]{Rigidity of four-dimensional\\ K\"ahler-Ricci  solitons}

\author{Xiaodong Cao}
\author{Ernani Ribeiro Jr}
\author{Hung Tran}

\address[X. Cao]{Department of Mathematics, Cornell University, Ithaca, NY 14853} \email{xiaodongcao@cornell.edu}
\address[E. Ribeiro Jr]{Departamento  de Matem\'atica, Universidade Federal do Cear\'a - UFC, Campus do Pici, 60455-760, Fortaleza - CE, Brazil}
\email{ernani@mat.ufc.br}

\address[H. Tran]{Department of Mathematics and Statistics, Texas Tech University, Lubbock, TX 79409} \email{hung.tran@ttu.edu}


\date{\today}

\begin{document}

\begin{abstract}
In this article, we investigate four-dimensional gradient shrinking Ricci solitons close to a K\"ahler model. The first theorem could be considered as a rigidity result for the K\"ahler-Ricci soliton structure on $\mathbb{S}^2\times \mathbb{R}^2$ (in the sense of Remark \ref{rmk1}). Moreover, we show that if the quotient of norm of the self-dual  Weyl tensor and scalar curvature is close to that on a K\"ahler metric in a specific sense, then the gradient Ricci soliton must be either half-conformally flat or locally K\"ahler. 

\end{abstract}

\maketitle

\section{Introduction}
\label{int}

A complete Riemannian metric $g$ on an $n$-dimensional smooth manifold $M^n$ is called a {\it gradient shrinking Ricci soliton} (GSRS) if there exists a smooth potential function $f$ such that the Ricci tensor $\Rc$ and metric $g$ satisfy the equation
\begin{equation}
\label{maineq}
\Rc+ Hess\,f=\frac{1}{2} g.
\end{equation} Here, $Hess\,f$ denotes the Hessian of $f$. Ricci solitons are self-similar solutions to Hamilton's Ricci flow and play a crucial role in the singularity analysis of the Ricci flow \cite{Hamilton2}. They also provide a natural extension of Einstein manifolds, see \cite{caoALM11} for an overview on Ricci solitons.

In \cite{Hamilton2}, Hamilton showed that any two-dimensional GSRS is either isometric to the plane $\Bbb{R}^2$ or a quotient of the sphere $\Bbb{S}^2$ (both are K\"ahler by the identifications $\Bbb{R}^2 \simeq \Bbb{C}$ and $\Bbb{S}^2 \simeq \Bbb{CP}^1$). Furthermore, it follows by the works of Ivey \cite{Ivey}, Perelman \cite{Perelman2}, Naber \cite{Naber}, Ni-Wallach \cite{Ni}, and H.-D. Cao-Chen-Zhu \cite{CaoA}, that any three-dimensional GSRS is a finite quotient of either the round sphere $\Bbb{S}^3$, the Gaussian shrinking soliton $\Bbb{R}^3,$ or the round cylinder $\Bbb{S}^{2}\times\Bbb{R}$. 

In dimension $n=4,$ the first non-Einstein example of compact gradient shrinking Ricci soliton is a rotationally K\"ahler metric on $\Bbb{CP}^2\sharp (- \Bbb{CP}^2)$, where $(-\Bbb{CP}^2)$ is the complex projective space with the opposite orientation. This was first discovered by H.-D. Cao \cite{Cao1} and Koiso \cite{Ko}. Later, Wang and Zhu \cite{WZ1} constructed a new gradient K\"ahler-Ricci soliton on $\Bbb{CP}^2\sharp 2(- \Bbb{CP}^2)$. A well-known folklore conjecture  asks {\it whether any compact non-Einstein gradient Ricci soliton in dimension four is necessarily K\"ahler}. For noncompact examples, we have the gradient K\"ahler-Ricci soliton $\Bbb{S}^2\times \Bbb{R}^2$ with  the natural orientation induced from the complex structure on $\Bbb{CP}^1 \times \Bbb{C}$ and the family constructed by Feldman, Ilmanen and Knopf in \cite{FIK}, which is $U(n)$-invariant and cone-like at infinity.

Despite of recent important advancements, the classification of four-dimensional GSRS remains open, even in the case of complete K\"ahler-Ricci solitons; see, e.g., \cite{CaoChen,CRZ,CZ2021,ELM, Ni, Chen,Chow,zhang, Naber, CZ, PW2, CWZ, FLGR, MW,KW, MS, MW2, cifarelli2022finite, BCCD22KahlerRicci} for recent progress. By the works \cite{ELM, Ni, zhang, PW2, CWZ, MS}, locally conformally flat (i.e., $\WW=0$) four-dimensional complete GSRS are isometric to finite quotients of either $\Bbb{S}^4$, or $\Bbb{R}^4$, or $\Bbb{S}^{3}\times\Bbb{R}.$ Other relevant classification results have been obtained under various curvature conditions:  half-conformally flat ($\WW^{+}=0$ by a choice of orientation) by Chen and Wang \cite{CW}; Bach-flat by H.-D. Cao and Chen \cite{CaoChen};  harmonic Weyl tensor ($\div \WW=0$) by Fern\'{a}ndez-L\'opez and Garc\'ia-R\'io \cite{FLGR} together with Munteanu and Sesum \cite{MS}; and harmonic self-dual Weyl tensor  ($\div \WW^{+}=0$) by Wu et al. \cite{Wu}, also see \cite{CMM}.

On the other hand, there are  results based on pinching conditions of the Weyl curvature. Catino \cite{Catino} showed that any complete four-dimensional GSRS with nonnegative Ricci curvature and 
\begin{equation}
\label{cond01}
|\WW| R\leq \sqrt{3}\left(|\mathring{Ric}|-\frac{1}{2\sqrt{3}}R \right)^2
\end{equation} must be conformally flat, here, $\mathring{Ric}$ and $R$ denote  traceless Ricci tensor and  scalar curvature, respectively. The nonnegative Ricci curvature assumption was later removed in \cite{Wu}. Some related results involving integral pinching conditions were established in \cite{CH,catinoAdv}. By a different method, Zhang \cite{Zhang2} shows that any four-dimensional GSRS with $0\leq Ric\leq C$ and 
\begin{equation}
\label{cond02}
|\WW|\leq \gamma \Big| |\mathring{Ric}|-\frac{1}{2\sqrt{3}}R\Big|
\end{equation} is either flat or has $2$-positive Ricci curvature, for some constant $\gamma<1+\sqrt{3}$. 

It turns out that neither pinching conditions (\ref{cond01}) nor (\ref{cond02})  recovers the non-Einstein K\"ahler-Ricci soliton $\Bbb{S}^2\times \Bbb{R}^2$, the (normalized) GSRS $\Big(\Bbb{S}^2 (\sqrt{2})\times \Bbb{R}^2,\,g,\,f\Big)$ with product metric $g$ and potential function \[f(x,y)=\frac{|y|^{2}}{4}+1,\] for $(x,y)\in\Bbb{S}^2 (\sqrt{2})\times \Bbb{R}^2.$ A recent result, due to H.-D. Cao, Ribeiro and Zhou \cite[Theorem 1]{CRZ}, states that a complete four-dimensional GSRS satisfying 
\begin{equation}
\label{cond10}
|\WW^{+}|^{2} -\sqrt{6}|\WW^{+}|^{3}\geq \frac{1}{2}\langle (\mathring{Ric}\odot \mathring{Ric})^{+},\WW^{+}\rangle
\end{equation} is either Einstein, conformally flat, or $\Bbb{S}^{2}\times\Bbb{R}^{2}$. Here, $\odot$ denotes the Kulkarni-Nomizu product. One important new feature in the above result is that they only need a condition on the self-dual part of Weyl tensor, but no point-wise Ricci curvature bound is needed.  Nonetheless, as mentioned in  \cite[Remark 2]{CRZ}, it is interesting to replace the right hand side of (\ref{cond10}) by a sharp expression depending on the norm of $\mathring{Ric}$ instead of a Kulkarni-Nomizu product. \\

In this article, we first prove a classification result, in the spirit of  \cite[Theorem 1]{CRZ}, using the norm of $\mathring{Ric}$. More precisely, we have established the following result.

\begin{theorem}
\label{thmA}
Let $(M^4,\,g,\,f)$ be a complete four-dimensional gradient shrinking Ricci soliton (\ref{maineq}) satisfying the pinching condition 
\begin{equation}
\label{cond1}
|\WW^{+}|^2-\sqrt{6}|\WW^+|^3 \geq \frac{\sqrt{6}}{6}|\mathring{Ric}|^2|\WW^+|.
\end{equation} Then $(M^4,\,g,\,f)$ is either 
\begin{enumerate}
\item[(i)] isometric to the standard round sphere $\Bbb{S}^4$,  the complex projective space $\Bbb{CP}^2,$ or
\item[(ii)] a finite quotient of the Gaussian shrinking soliton $\Bbb{R}^4$,  the round cylinder $\Bbb{S}^{3}\times\Bbb{R},$ or 
\item[(iii)] a compact Einstein manifold with $\nabla \WW^+ \equiv 0$ and $\WW^+$ has precisely two distinct eigenvalues (including $\Bbb{S}^2 \times \Bbb{S}^2$ with the product metric), or
\item[(iv)] a finite quotient of the K\"ahler-Ricci soliton $\Bbb{S}^2\times \Bbb{R}^2.$ 
\end{enumerate}
\end{theorem}

A key ingredient in the proof of Theorem \ref{thmA} is the novel sharp estimate 
\begin{equation}
	\label{epkl1}
	\langle (\mathring{Ric}\odot \mathring{Ric})^{+},\WW^{+}\rangle \leq \frac{\sqrt{6}}{3}|\mathring{Ric}|^{2}|\WW^{+}|.
\end{equation} Interestingly, its proof is based on the notion of the curvature of the second kind and its algebraic consequences outlined by X. Cao, Gursky and Tran \cite{CGT} in the resolution of a conjecture of Nishikawa.

\begin{remark} \label{rmk1}
We point out that the equality in Theorem \ref{thmA} holds for $\Bbb{S}^2\times \Bbb{R}^2:$ \[|\WW^{+}|^2-\sqrt{6}|\WW^+|^3 =\frac{\sqrt{6}}{6}|\mathring{Ric}|^2|\WW^{+}|=\frac{1}{48}.\] Thereby, Theorem \ref{thmA} can be seen as a gap theorem for $\Bbb{S}^2 \times \Bbb{R}^2.$
\end{remark}

\begin{remark}
A relevant observation is that, by using essentially the same arguments as in the proof of Theorem \ref{thmA}, if we replace the assumption (\ref{cond1}) by the condition 

\begin{equation}
\label{cond12}
|\WW^{+}|-\sqrt{6}|\WW^+|^2 \geq \frac{\sqrt{6}}{6}|\mathring{Ric}|^2.
\end{equation} we derive the same classification of Theorem \ref{thmA} with the exception of the round cylinder $\Bbb{S}^{3}\times\Bbb{R}.$ 

\end{remark}

In our next result, we establish a general rigidity theorem for four-dimensional K\"ahler-Ricci solitons. It is known from \cite{derd1} that a four-dimensional K\"ahler metric with a natural of orientation must satisfy
\[|\WW^+|^2=\frac{R^2}{24}.\] 
We will show that on a four-dimensional GSRS, if the quotient $\frac{|\WW^+|^2}{R^2}$ is close enough to $1/24$  from below, then the manifold must be either locally K\"ahler or one of the standard models. To be precise, we have the following result.

\begin{theorem}
\label{thmB}
Let $(M^4,\,g,\,f)$ be a four-dimensional complete (nonflat) gradient shrinking Ricci soliton (\ref{maineq}) satisfying the pinching condition 
\begin{equation}
\label{cond2}
\frac{1}{24} \geq \frac{|\WW^{+}|^2}{R^2} \geq \frac{\sqrt{6}}{3}\frac{|\mathring{Ric}|^2}{R^3}|\WW^{+}|.
\end{equation} Then $(M^4,\,g,\,f)$ is either
\begin{enumerate}
\item[(i)] isometric to $\Bbb{S}^4,$ or $\Bbb{CP}^2,$ or
\item[(ii)] a finite quotient of the round cylinder $\Bbb{S}^{3}\times\Bbb{R},$ or
\item[(iii)] a locally K\"ahler-Ricci soliton.
\end{enumerate}
\end{theorem}

\begin{remark}
{\it Locally K\"ahler} in Theorem \ref{thmB} means K\"ahler after possibly pulling back to a double cover of $M^4.$
\end{remark}
\begin{remark}
We highlight that the condition in Theorem \ref{thmB} is sharp as both inequa\-li\-ties in (\ref{cond2}) become equalities for the K\"ahler-Ricci soliton $\Bbb{S}^2\times \Bbb{R}^2$. 
\end{remark} 

In the compact case, the reverse of (\ref{cond2}) leads to the following characterization.

\begin{corollary}
\label{corC}
Let $(M^4,\,g,\,f)$ be a four-dimensional compact gradient shrinking Ricci soliton (\ref{maineq}) satisfying 
\begin{equation}
\label{cond3}
\frac{1}{24} \leq \frac{|\WW^{+}|^2}{R^2} \leq \frac{\sqrt{6}}{3}\frac{|\mathring{Ric}|^2}{R^3}|\WW^{+}|.
\end{equation} Then $(M^4,\,g,\,f)$ is a locally K\"ahler-Ricci soliton.
\end{corollary}

\begin{remark}
	In the compact case, it is known (see \cite{caoALM11,WZ1}) that a nontrivial K\"ahler-Ricci soliton is Fano (i.e., the first Chern class is positive) and the Futaki-invariant is nonzero. Moreover, it follows from the works of Tian and Zhu \cite{TZ1,TZ2} that the soliton vector field is unique up to holomorphic automorphisms of the underlying complex manifold.
\end{remark}

This article is organized as follows. In Section \ref{Background}, we review some basic facts and useful lemmas on four-dimensional GSRS. Section \ref{keyestimates} contains relevant contributions with several novel estimates. Finally, Section \ref{rigidityR} collects the proofs of our main results. \\

{\bf Acknowledgment.} X. Cao was partially supported by the Simons Foundation (\#585201). E. Ribeiro was partially supported by CNPq/Brazil (\#309663/2021-0), CAPES/Brazil and FUNCAP/Brazil (\# PS1-0186-00258.01.00/21). H. Tran was partially supported by a Simons Collaboration Grant and NSF grant DMS-2104988. The authors would like to thank Detang Zhou for helpful comments on a preliminary version of the paper.\\
 
\section{Background}
\label{Background}

In this section, we review some basic facts and present key results that will be used for our main theorems. 

Throughout this paper, for an $n$-dimensional Riemannian manifold $(M^n,\,g),$ we denote $\Rc$, $\mathring{\Rc},$ $R$ and $K$ to be the Ricci, traceless Ricci, scalar and sectional curvatures, respectively. Given a point $p\in M$, let $\{e_1,\,\ldots\, ,e_n\}$ be a local normal orthogonal coordinate of $T_pM,$ then $\{e^1,\,\ldots\,,e^n\}$ denotes the associated dual basis. 

For a finite dimensional vector space $V$, ${S}^2 (V)$ and $\Lambda^2(V)$ are the space of symmetric and anti-symmetric linear maps on $V$ (called symmetric $2$-tensors and $2$-forms, respectively). In particular, $S^2_0(V)$ contains only traceless symmetric maps. When $V=T_pM$, we normally subdue the vector space for convenience. The convention for the inner products is as follows. For $u, v\in S^2(V)$ and $\alpha, \beta\in \Lambda^2(V)$:
	\begin{align*} 
	\left\langle{u, v}\right\rangle &= \text{Tr}\,(u^Tv)\,\,\,\,\hbox{and}\,\,\,\,
	~~~~\left\langle{\alpha, \beta}\right\rangle = \frac{1}{2}\text{Tr}\,(\alpha^T\beta).
	\end{align*}
	
Let $\mathcal{R}(V)$ be the space of algebraic curvatures, that is, $(4,0)$-tensors satisfying certain symmetric properties and the first Bianchi identity. More precisely, for $T\in \mathcal{R}(V),$ we have 
	\begin{align*}
	T(e_i, e_j, e_k, e_l) = -T(e_j, e_i, e_k, e_l)=-T(e_i, e_j, e_l, e_k)= T(e_k, e_l, e_i, e_j)
	\end{align*} and
	
		\begin{align*}
	0 =T(e_i, e_j, e_k, e_l)+T(e_i, e_k, e_l, e_j)+T(e_i, e_l, e_j, e_k).
	\end{align*}

Thus, any element $T\in \mathcal{R}(V)$ can be considered as an element in either $\text{End}(\Lambda^2(V))$ or $\text{End}(S^2(V))$. They are so called the \textit{curvature of the first} or \textit{second kind}, respectively (see \cite{CGT} and references therein). Precisely, for $\omega\in \Lambda^2(V)$ and $A\in S^2(V),$ we define
\begin{align*}
	T(\omega)(e_i, e_j)&:=\sum_{k<l}T(e_i, e_j, e_k, e_l)\omega (e_k, e_l)
	\end{align*}
	and 
	\begin{align*}
	(\hat{T}A)(e_i, e_k)&:=\sum_{j,l} T(e_i, e_j, e_l, e_k)A(e_j, e_l).
	\end{align*}
	The inner product in $\text{End}(\Lambda^2(V))$ is given by 
	\begin{align*}
	\left\langle{T_1, T_2}\right\rangle &= \sum_{i<j, k<l}T_1(e_{ij}, e_{kl})T_2(e_{ij}, e_{kl})
=\frac{1}{4}\sum_{i,j, k,l} (T_1)_{ijkl} (T_2)_{ijkl}. 
	\end{align*}

	Recall that the Kulkarni-Nomizu product $\odot$: $S^2(V)\times S^2(V)\mapsto \mathcal{R}(V)$ is defined by
\begin{eqnarray}
\label{eq76}
(A \odot B)_{ijkl}= A_{ik}B_{jl}+A_{jl}B_{ik}-A_{il}B_{jk}-A_{jk}B_{il}.
\end{eqnarray} 
		
In terms of the Kulkarni-Nomizu product, we have the following intrinsic relation, its proof follows from a straightforward calculation.

\begin{lemma}
	\label{knproduct}
	Let $A, B\in S^2(V)$, then we have:
\[(\hat{T}A,B)=-\left\langle{T,A\odot B}\right\rangle. \]
\end{lemma}

We further recall  the following curvature decomposition
\begin{eqnarray}
\label{weyl}
R_{ijkl}&=&\WW_{ijkl}+\frac{1}{n-2}\big(R_{ik}g_{jl}+R_{jl}g_{ik}-R_{il}g_{jk}-R_{jk}g_{il}\big) \nonumber\\
 &&-\frac{R}{(n-1)(n-2)}\big(g_{jl}g_{ik}-g_{il}g_{jk}\big),
\end{eqnarray} where $R_{ijkl}$ stands for the Riemann curvature tensor and $W_{ijkl}$ is the Weyl curvature tensor.

\subsection{Four-dimensional Manifolds}

In this subsection, we focus on dimension $n=4.$ The bundle of $2$-forms $\Lambda^2$ can be invariantly decomposed into a direct sum, 
\begin{equation}
\label{lk1}
\Lambda^2=\Lambda^{+}\oplus\Lambda^{-}.
\end{equation} 
Moreover, as observed in \cite{ahs78} (see also \cite[Lemma 2]{derd1}), these components have special structures: 

\begin{lemma}
	\label{lambda2pm}
	Let $(M^4,\,g)$ be an oriented four-dimensional Riemannian manifold and $p$ be a point in $M^4.$ Then the following assertions hold:
	\begin{enumerate}
		\item $\Lambda^{\pm}$ are mutually commuting, each isomorphic to $\mathfrak{so}(3)$. 
		\item Elements of lengths $\sqrt{2}$ in $\Lambda^\pm$ coincides with the almost complex structure compatible with the metric. 
		\item Each oriented orthogonal basis $\{ \omega_1, \omega_2, \omega_3\}$ of $\Lambda^{\pm}$ with $|\omega_1|=|\omega_2|=|\omega_3|=\sqrt{2}$ forms a quaternionic structure in $T_pM$. That is, $$\omega_1^2=\omega_2^2=\omega_3^2=-\text{I} ~and ~\omega_1\omega_2=\omega_3=-\omega_2\omega_1.$$
		\item Given $\omega\in \Lambda^+$, $\alpha\in \Lambda^-$, $\omega\alpha=\alpha\omega$ is an orientation preserving involution of $T_pM$. Its $(\pm 1)$-eigenspaces form an orthogonal decomposition of $T_pM$ into a direct sum of two planes. In particular, we have $\omega\alpha\in S^2_0$. 		
	\end{enumerate}
\end{lemma}

The decomposition (\ref{lk1}) is in particular conformally invariant and induces a decomposition for the Weyl curvature: 
\begin{equation}
\label{ewq}
\WW = \WW^+\oplus \WW^-,
\end{equation} where $\WW^\pm:\Lambda^\pm \longrightarrow\Lambda^\pm $ are called the {\it self-dual} part and {\it anti-self-dual} part of the Weyl tensor, respectively. Furthemore, at a point $p\in M^4$, one can diagonalize $\WW^\pm$ with eigenforms $\{\omega_i\}_{i=1}^3 \in \Lambda^+$ and $\{\alpha_i\}_{i=1}^3 \in \Lambda^-$ such that $\lambda_i$ and $\mu_{i},$ $1\le i \le 3,$ are the respective eigenvalues. In particular, one observes that
\begin{equation}
\label{eigenvalues}
\begin{cases}
 \lambda_1\geq \lambda_2 \geq \lambda_3 \,\,\,  \hbox{and}\,\,\, \lambda_1 +\lambda_2 +\lambda_3= 0, &\\
\mu_1\geq \mu_2 \geq \mu_3\,\,\,  \hbox{and}\,\,\,u_1 +\mu_2 +\mu_3= 0.&
\end{cases}
\end{equation} 
So, it follows that 
\begin{equation}
	\label{eigenW}2\WW=\sum_{i=1}^3\lambda_i \omega_i\otimes \omega_i+\mu_i \alpha_i\otimes \alpha_i.\end{equation}
In a local neighborhood, we construct a local frame so that (\ref{eigenW}) holds. In par\-ti\-cular, $\Lambda^{\pm}$ is invariant under parallel displacement. By using Derdzinski's argument \cite[p. 414]{derd1}, we arrive at, for  $(i, j, k)$ an orientation-preserving permutation of $(1, 2, 3)$,
\begin{align}
	\nabla \omega_i &= a_j \otimes\omega_k - a_k \otimes\omega_j, \nonumber \\
\label{nablaW}
	2\nabla \WW^+ &= \sum_{i=1}^3 \left(d\lambda_i \otimes \omega_i+ (\lambda_i-\lambda_k)a_j \otimes \omega_k - (\lambda_i-\lambda_j)a_k \otimes\omega_j\right)\otimes \omega_i.
\end{align}

The following Kato inequality will be also useful.
\begin{lemma}
	\label{Katoinequality}
	Let $(M^{4},\,g)$ be an oriented four-dimensional Riemannian ma\-ni\-fold. Then we have:
	\[|\nabla |\WW^+|| \leq |\nabla \WW^+|.\]
	Equality holds if and only if there is an one-form $\nu$ such that
	$\nabla \WW^+=\nu \otimes \WW^+.$ 
\end{lemma}

We will also need to use the following algebraic inequality 
\begin{equation}
	\label{eqdet}
	\det \WW^{+}\leq \frac{\sqrt{6}}{18}|\WW^{+}|^{3},
\end{equation} moreover, equality holds if and only if $\lambda_{3}=\lambda_{2}=-\frac{1}{2}\lambda_{1}.$ 

\subsection{Two-forms and Two-tensors} In dimension four, there is an interesting connection between $2$-tensors and $2$-forms. A choice of orthonormal bases of $\Lambda^{\pm}$ is given by
\begin{align}
\label{Hodgebase}
\mathbb{B}^+ &=\frac{1}{{\sqrt{2}}}(e_{12}+e_{34}, e_{13}-e_{24}, e_{14}+e_{23})=\frac{1}{{\sqrt{2}}}(\omega_1, \omega_2, \omega_3),\\
\mathbb{B}^- &=\frac{1}{{\sqrt{2}}}(e_{12}-e_{34}, e_{13}+e_{24}, e_{14}-e_{23})=\frac{1}{{\sqrt{2}}}(\alpha_1, \alpha_2, \alpha_3).\nonumber
\end{align} By Lemma \ref{lambda2pm}, their cross products form a basis for $S^2_0(V),$ i.e., 
\begin{equation}
\label{basis2tensordim4}
\mathbb{B}_2=\{\omega_i \alpha_j\}=\{h_1,\ldots,\,h_9\}.
\end{equation} 
More specifically, one has

\begin{minipage}[c]{0.49\linewidth}
\centering
\begin{equation*}
h_1=\omega_1\alpha_1= \left( \begin{array}{cccc}
-1 & & &   \\
& -1 & & \\
& & 1 &\\
& & & 1 
\end{array} \right);
\end{equation*}
\end{minipage}
\hfill 
\begin{minipage}[c]{0.49\linewidth}
\centering
\begin{equation*}
h_2=\omega_1\alpha_2= \left( \begin{array}{cccc}
& & & 1   \\
& & -1 & \\
& -1 & &\\
1 & & & 
\end{array} \right);
\end{equation*}
\end{minipage}

\begin{minipage}[c]{0.19\linewidth}
\centering
\begin{equation*}
h_3=\omega_1\alpha_3=\left( \begin{array}{cccc}
& & -1 &   \\
&  & & -1 \\
-1 & & &\\
& -1 & &  
\end{array} \right);
\end{equation*}
\end{minipage}
\hfill
\begin{minipage}[c]{0.46\linewidth}
\centering
\begin{equation*}
\,\,\,h_4=\omega_2\alpha_1= \left( \begin{array}{cccc}
& & & -1  \\
& & -1 & \\
& -1 & &\\
-1 & & &  
\end{array} \right);
\end{equation*}
\end{minipage}

\begin{minipage}[c]{0.49\linewidth}
\centering
\begin{equation*}
h_5=\omega_2\alpha_2= \left( \begin{array}{cccc}
-1 & & &   \\
& 1 & & \\
& & -1 &\\
& & & 1 
\end{array} \right);
\end{equation*}
\end{minipage}
\hfill
\begin{minipage}[c]{0.49\linewidth}
\centering
\begin{equation*}
h_6=\omega_2\alpha_3= \left( \begin{array}{cccc}
& 1& &   \\
1 & & & \\
& & & -1\\
& & -1 &  
\end{array} \right);
\end{equation*}
\end{minipage}

\begin{minipage}[c]{0.49\linewidth}
\centering
\begin{equation*}
h_7=\omega_3\alpha_1= \left( \begin{array}{cccc}
& & 1 &   \\
&  & & -1 \\
1 & & &\\
& -1 & &  
\end{array} \right);
\end{equation*}
\end{minipage}
\hfill
\begin{minipage}[c]{0.49\linewidth}
\centering
\begin{equation*}
h_8=\omega_3\alpha_2= \left( \begin{array}{cccc}
& -1 & &   \\
-1 & & & \\
& & & -1\\
& & -1 & 
\end{array} \right);
\end{equation*}
\end{minipage}

\begin{minipage}[c]{0.49\linewidth}
\centering
\begin{equation*}
h_9=\omega_3\alpha_3= \left( \begin{array}{cccc}
-1 & & &   \\
& 1 & & \\
& & 1 &\\
& & & -1 
\end{array} \right).
\end{equation*}
\end{minipage}

At the same time, according to Berger \cite{berger61} (see also \cite{st69}), at a point, there is a normal form for the Weyl tensor: a basis $\{e_i\}_{i=1}^4$ such that the bases (\ref{Hodgebase}) consist of eigenforms of $\WW^{\pm}$. Whence, X. Cao, Gursky and Tran in \cite[Proposition 4.3]{CGT} are able to compute $\widehat{\WW}$ with respect to (\ref{basis2tensordim4}), 
\begin{align} \label{Wmatrix}
\widehat{\WW} = \left (\begin{array}{lll}
\, \mathcal{D}_1 \, & \, 0 \, & \, 0 \, \\
\, 0 \, & \, \mathcal{D}_2 \, & \, 0  \, \\
\, 0 \, & \, 0 \, & \, \mathcal{D}_3 \,
\end{array}
\right),
\end{align}
here the $\mathcal{D}_i$'s are diagonal matrices given by

\begin{align*} \label{D1matrix}
\mathcal{D}_i =  \left (\begin{array}{lll}
\displaystyle -4( \lambda_i + \mu_1)  &  &  \\
& \displaystyle -4 ( \lambda_i + \mu_2 )    &  \\
&   &  \displaystyle -4 ( \lambda_i + \mu_3)
\end{array}
\right).
\end{align*}

Before proceeding, we list the curvature operator of the {\it second kind} $\widehat{R}$ for some basic examples:

\begin{enumerate}
\item[{\bf 1.}] Let $(\Bbb{S}^4(\sqrt{6}),\,g_{0})$ be the $4$-dimensional round sphere of radius $\sqrt{6}$ and scalar curvature $R=2.$ In this case, one has $$\widehat{R}=\sqrt{6}\Bbb{I},$$ where $\Bbb{I}$ is the identity matrix.
\item[\textbf{2.}] Let $(\Bbb{CP}^2,\,g_{_{FS}})$  be the complex projective space of complex dimension $2$ with the Fubini-Study metric and scalar curvature $R=8.$ Then, up to the ordering of eigenvalues, one has

\begin{align*}
	\widehat{R} = 16 \left (\begin{array}{lll}
		-\frac{1}{2} \mathbb{I} \, & \, 0 \, & \, 0 \, \\
		\ \ \ 0 \, & \, \mathbb{I} \, & \, 0 \, \\
		\ \ \ 0 \, & \, 0 \, & \, \mathbb{I} \,
	\end{array}
	\right).
	\end{align*}
	
	\item[\textbf{3.}] Let $\Big(\Bbb{S}^2 (\sqrt{2})\times \Bbb{R}^2,\,g,\,f\Big)$ be the GSRS with the product metric $g$ and scalar curvature $R=1.$ In this case, up to the ordering of eigenvalues, one has
	\begin{align*}
	\widehat{R} = \sqrt{2} \left (\begin{array}{lllllllll}
		-\frac{1}{2} & \ & \ & \ & \ & \ & \ & \ & \ \\
		\ & 1 & \ & \ & \ & \ & \ & \ & \ \\
		\ & \ & 1 & \ & \ & \ & \ & \ & \ \\
		\ & \ & \ & 0 & \ & \ & \ & \ & \ \\
		\ & \ & \ & \ & 0 & \ & \ & \ & \ \\
		\ & \ & \ & \ & \ & 0 & \ & \ & \ \\
		\ & \ & \ & \ & \ & \ & 0 & \ & \ \\
		\ & \ & \ & \ & \ & \ & \ & 0 & \ \\
		\ & \ & \ & \ & \ & \ & \ & \ & 0 \\
	\end{array}
	\right).
\end{align*}
	
	\end{enumerate}

\subsection{Four-Dimensional Shrinking Ricci Solitons}
In this subsection, we are going to collect some well-known identities for four-dimensional GSRS satisfying (\ref{maineq}). First, let us recall the following lemma. 

\begin{lemma}[\cite{Hamilton2}]
\label{lem1}
Let $\big(M^4,\,g,\,f\big)$ be a four-dimensional gradient shrinking Ricci soliton satisfying (\ref{maineq}). Then we have:
\begin{enumerate}
\item $R+\Delta f=2;$
\item $\frac{1}{2}\nabla R=Ric(\nabla f);$
\item $\Delta_{f} R= R-2|Ric|^{2};$
\item $R+|\nabla f|^{2}= f$ (after normalizing);
\item $\Delta_{f} R_{ij}=R_{ij}-2R_{ikjl}R_{kl}.$
\end{enumerate} Here, $\Delta_f\cdot:=\Delta\cdot-\nabla_{\nabla f}\cdot$ denotes the drifted Laplacian. 
\end{lemma}

Chen showed in \cite{Chen} that any complete ancient solution to the Ricci flow has nonnegative scalar curvature, it follows that $R\geq 0$ for any complete GSRS. Moreover, $R$ is strictly positive  unless $(M^4,\,g,\,f)$ is the Gaussian shrinking soliton (see \cite{PRS}). 

Regarding the potential function $f$, H.-D. Cao and Zhou \cite{CZ} proved that 
\begin{equation}
\label{eqfbeh}
\frac{1}{4}\Big(r(x)-c\Big)^{2}\leq f(x)\leq \frac{1}{4}\Big(r(x)+c\Big)^{2},
\end{equation} for all $r(x)\geq r_{0},$ where $r=r(x)$ is the distance function to a fixed point in $M.$ Additionally, they showed that every complete noncompact GSRS has at most Euclidean volume growth (see \cite[Theorem 1.2]{CZ}). These asymptotic estimates are optimal in the sense that they are achieved by the Gaussian shrinking soliton.

Next, we recall a Weitzenb\"ock type formula established by the first and third authors \cite{CH} that will play a crucial role in our proofs.

\begin{proposition}[\cite{CH}]
\label{propB}
Let $(M^4,\,g,\,f)$ be a four-dimensional gradient shrinking Ricci soliton satisfying (\ref{maineq}). Then we have:
$$\Delta_{f} |\WW^{\pm}|^{2}=2|\nabla \WW^{\pm}|^{2}+2|\WW^{\pm}|^{2}-36\, \det \WW^{\pm}-\langle (\mathring{Ric}\odot \mathring{Ric})^{\pm},\WW^{\pm}\rangle,$$
where $\odot$ stands for the Kulkarni-Nomizu product.
\end{proposition}

\section{Key Estimates}
\label{keyestimates}
In this section, we will establish  some key lemmas that will be used in  our proofs.

\begin{lemma}
	\label{kahlerform}
		Let $(M^4,\ g)$ be a four-dimensional Riemannian manifold. If $\nabla \WW^+= \nu \otimes \WW^+$ for some one-form $\nu$ and $\lambda_2=\lambda_3$ at each point, then $\omega_1$ is a locally K\"ahler form.
\end{lemma}

\begin{proof}
	Since  $\lambda_2=\lambda_3$ at each point, from (\ref{eigenvalues}), it is clear  that
	\[\lambda_1=-2\lambda_2=-2\lambda_3\]
	is a non-negative function. Its eigenspace $\omega_1$ is therefore locally defined (for example, see \cite{derd1}). Hence, it follows from (\ref{nablaW}) that 
	\begin{eqnarray*}
		2\nabla \WW^+ &=& (d\lambda_1\otimes \omega_1 +\frac{3}{2}\lambda_1 a_2\otimes \omega_3-\frac{3}{2}\lambda_1 a_3\otimes \omega_2)\otimes \omega_1\\
		&& + (d\lambda_2\otimes \omega_2 -\frac{3}{2}\lambda_1 a_3\otimes \omega_1)\otimes \omega_2\\
		&& + (d\lambda_3\otimes \omega_3 +\frac{3}{2}\lambda_1 a_2\otimes \omega_1)\otimes \omega_3.
	\end{eqnarray*} 
	At the same time, since $\nabla \WW^+= \nu \otimes \WW^+,$ we have
	\begin{align*}
		2\nabla \WW^+ &= \sum_{i=1}^3\lambda_i\nu\otimes \omega_i\otimes \omega_i.  
	\end{align*} The equations above then imply that 
	\begin{align*}
		\lambda_1\nu &= d\lambda_1,\\
		a_2 &= a_3\equiv 0. 
	\end{align*} Consequently, $\nabla \omega_1\equiv 0$ and, since $\omega_1\in \Lambda^+,$ it is a locally K\"ahler form. 
\end{proof}

Next, we obtain a sharp estimate for $\langle (\mathring{Ric}\odot \mathring{Ric})^{+},\WW^{+}\rangle$.

\begin{lemma}
	\label{lemK}
	Let $(M^{4},\,g)$ be an oriented four-dimensional Riemannian manifold. Then we have:     	
	\begin{equation}
		\label{EqHu}
		\langle (\mathring{Ric}\odot \mathring{Ric})^{+},\WW^{+}\rangle \leq \frac{\sqrt{6}}{3}|\mathring{Ric}|^{2}|\WW^{+}|.
	\end{equation} Moreover, equality holds if and only if $\WW^+$ has eigenvalues 
	\[0\leq \lambda_1=-2 \lambda_2=-2\lambda_3\]
	and 
	\[\mathring{\Rc}=a_1 h_1+ a_2 h_2+ a_3 h_3. \]
\end{lemma}

\begin{proof}
	A priori, using the orthogonal basis (\ref{basis2tensordim4}), one has $\mathring{\Rc}=\sum_{i=1}^9 a_i h_i$. It follows from (\ref{Wmatrix}) that
	\begin{align*}
		\widehat{\WW^+}(\mathring{\Rc}, \mathring{\Rc}) &= a_i^2\widehat{\WW^+}(h_i, h_i)\\
		&= -4\sum_{i=1}^3\lambda_i\sum_{j=3i-2}^{3i}a_j^2.
	\end{align*} Denote $A_i^{2}=\sum_{j=3i-2}^{3i}a_j^2$ and hence, since $|h_i|=2$, one obtains that
	\[|\mathring{\Rc}|^2= 4\sum_{i}A_i^2.\]
	Without loss of generality, we may consider $|\mathring{\Rc}|^2= \rho\geq 0,$ $|\WW^+|=\sigma\geq 0$ and
	\begin{align*}
		\vec{\lambda} &=(\lambda_1, \lambda_2, \lambda_3),\\
		\vec{A} &=(A_1^2, A_2^2, A_3^2).
	\end{align*}
	Therefore, $\vec{\lambda}$ is on the circular intersection of the sphere $x^2+y^2+z^2=\sigma^2$ and region $x\geq y \geq z$ in the plane $x+y+z=0.$ Besides, $\vec{A}$ is on the triangular intersection of the plane $x+y+z=\rho/4$ and the quadrant $x, y, z\geq 0$. Consequently,
	\begin{equation}
	\label{2qw10}
		-\widehat{\WW^+}(\mathring{\Rc}, \mathring{\Rc}) =4\sum_{i=1}^3 \lambda_i A_i^2= 4\left\langle{\vec{\lambda}, \vec{A}}\right\rangle= 4\left\langle{\vec{\lambda}, \vec{\overline{A}}}\right\rangle.
	\end{equation}
	Here, $\vec{\overline{A}}$ is the projection of $\vec{A}$ on the plane $x+y+z=0,$ that is, 
\begin{equation}
\label{3qw10}
\vec{\overline{A}}=\frac{1}{3}\left(2A_{1}^{2}-A_{2}^{2}-A_{3}^{2},-A_{1}^{2}+2A_{2}^{2}-A_{3}^{2},-A_{1}^{2}-A_{2}^{2}+2A_{3}^{2}\right).
\end{equation}
Observe that if $\vec{\overline{A}}$ is $(0,0,0),$ then it suffices to use (\ref{2qw10}) and (\ref{3qw10}) in order to infer that the asserted inequality is trivially satisfied. Therefore, from the description of $\vec{A},$ one can deduce that $\vec{\overline{A}}$ is within an equilateral triangle centered at the origin and vertices $(\rho/6,-\rho/12,-\rho/12),$ $(-\rho/12,\rho/6,-\rho/12)$ and $(-\rho/12,-\rho/12,\rho/6).$ In particular, the maximum value of  $\langle{\vec{\lambda}, \vec{\overline{A}}}\rangle$ is attained if and only if $\vec{\overline{A}}$ coincides with a vertex of the triangle and $\vec{\lambda}$ is parallel to it. Thereby, $\vec{A}$ coincides with a vertex and $\vec{\lambda}$ is a positive multiple of the projection of $\vec{A}$ on the plane $x+y+z=0$. By a direct computation at three vertices, one sees that the maximum is achieved at  
	\[\vec{A}=(\rho/4, 0, 0).\]
	Then, $\vec{\lambda}$ is a positive multiple of the projection of $\vec{A}$ on the plane $x+y+z=0$ and hence, it is not difficult to check that $\vec{\lambda}$ must be a multiple of $(2, -1, -1).$ Of which, we obtain 
		\begin{equation}
		\label{eq45t}
		-\widehat{\WW^+}(\mathring{\Rc}, \mathring{Rc}) \leq \frac{2}{\sqrt{6}}|\WW^+||\mathring{\Rc}|^2.
	\end{equation} Furthermore, equality holds if and only if $\WW^+$ has eigenvalues 
	\[0\leq \lambda_1=-2 \lambda_2=-2\lambda_3\]
	and 
	\[\mathring{\Rc}=a_1 h_1+ a_2 h_2+ a_3 h_3. \] 	
	Finally, it suffices to use Lemma \ref{knproduct} and (\ref{eq45t}) to calculate that
	\[\left\langle{\WW^+, \mathring{\Rc}\odot \mathring{\Rc}}\right\rangle= -\widehat{\WW^+}(\mathring{\Rc}, \mathring{Rc})\leq \frac{2}{\sqrt{6}}|\WW^+||\mathring{\Rc}|^2,\] which gives the desired result.
\end{proof}

To conclude this section, we shall use Proposition \ref{propB} to establish a lemma that will be employed in the proof of Theorem \ref{thmB}. 

\begin{lemma}
	\label{lemB}
	Let $(M^4,\,g,\,f)$ be a four-dimensional gradient shrinking Ricci soliton. Then, for $\Psi=f-2\ln R,$ we have
	\begin{eqnarray}
		\Delta_{\Psi}\left(\frac{|W^+|^2}{R^2}\right)&=& 4\frac{|Ric|^2}{R^3}|W^+|^2+2\frac{|\nabla R|^2}{R^4}|W^+|^2+\frac{2}{R^2}|\nabla W^+|^2\nonumber\\&& -36\frac{\det W^+}{R^2}-\frac{1}{R^2}\langle (\mathring{Ric}\odot \mathring{Ric})^{+},W^{+}\rangle\\&&-\frac{2}{R^3}\langle \nabla |W^{+}|^{2},\nabla R\rangle.\nonumber
	\end{eqnarray}
	
\end{lemma}
\begin{proof}
	First, one observes that
	
	\begin{equation}
		\label{eq1a1}
		\Delta\left(\frac{|W^{+}|^2}{R^2}\right)=R^{-2}\Delta |W^{+}|^{2}+|W^{+}|^{2}\Delta \left(R^{-2}\right)+2\langle \nabla |W^+|^2 , \nabla \left( R^{-2}\right)\rangle.
	\end{equation} Moreover,  we have $\nabla (R^{-2})=-2R^{-3}\nabla R$ so that
		\begin{equation*}
		\Delta (R^{-2})=-2R^{-3}\Delta R+6R^{-4}|\nabla R|^{2}.
	\end{equation*} This substituted into (\ref{eq1a1}) gives
	
	\begin{eqnarray*}
		\Delta \left(\frac{|W^+|^2}{R^2}\right)&=& |W^+|^{2}\left(-2\frac{\Delta R}{R^3}+6\frac{|\nabla R|^{2}}{R^4}\right) +\frac{1}{R^{2}}\Delta |W^{+}|^2\nonumber\\&& -\frac{4}{R^3}\langle \nabla |W^{+}|^{2},\nabla R\rangle.
	\end{eqnarray*} Now, we may use Proposition \ref{propB} and Lemma \ref{lem1} to infer
		\begin{eqnarray*}
		\Delta \left(\frac{|W^+|^2}{R^2}\right)&=& |W^+|^{2}\left[-\frac{2}{R^3}\left(\langle \nabla R,\nabla f\rangle +R-2|Ric|^{2}\right)+6\frac{|\nabla R|^{2}}{R^4}\right]\nonumber\\&&+\frac{1}{R^{2}}\Big[\langle \nabla f,\nabla |W^{+}|^{2}\rangle +2|\nabla W^{+}|^{2}+2|W^{+}|^{2}\nonumber\\&&-36\,\det W^{+}-\langle (\mathring{Ric}\odot \mathring{Ric})^{+},W^{+}\rangle\Big]\nonumber\\&&-\frac{4}{R^3}\langle \nabla |W^{+}|^{2},\nabla R\rangle.\nonumber
	\end{eqnarray*} Consequently, 
		\begin{eqnarray}
		\label{eq1a2}
		\Delta \left(\frac{|W^+|^2}{R^2}\right)&=&-2\frac{|W^+|^2}{R^3}\langle \nabla R,\nabla f\rangle + 4\frac{|Ric|^2}{R^3}|W^{+}|^{2}+6\frac{|\nabla R|^2}{R^4}|W^+|^2\nonumber\\&&+ \frac{1}{R^2}\langle \nabla f,\nabla |W^+|^{2}\rangle +2\frac{|\nabla W^+|^2}{R^2}-36\,\frac{\det W^+}{R^2}\nonumber\\&&-\frac{1}{R^2}\langle (\mathring{Ric}\odot \mathring{Ric})^{+},W^{+}\rangle-\frac{4}{R^3}\langle \nabla |W^{+}|^{2},\nabla R\rangle.
	\end{eqnarray}

	On the other hand, observe that
	
	$$\left\langle \nabla \Big(\frac{|W^+|^2}{R^2}\Big),\nabla f\right\rangle=\frac{1}{R^2}\langle \nabla |W^{+}|^{2},\nabla f\rangle-2\frac{|W^{+}|^2}{R^3}\langle \nabla R,\nabla f\rangle.$$ Therefore,  (\ref{eq1a2}) becomes

	\begin{eqnarray}
		\label{eq1a3}
		\Delta_{f} \left(\frac{|W^+|^2}{R^2}\right)&=&4\frac{|Ric|^2}{R^3}|W^+|^2 + 6\frac{|\nabla R|^2}{R^4}|W^+|^{2}+2\frac{|\nabla W^+|^2}{R^2}\nonumber\\&&-36\frac{\det W^+}{R^2} -\frac{1}{R^2}\langle (\mathring{Ric}\odot \mathring{Ric})^{+},W^{+}\rangle\\&&-\frac{4}{R^3}\langle \nabla |W^{+}|^{2},\nabla R\rangle.\nonumber
	\end{eqnarray} Now, choosing $\varphi=2\ln R$ we have 
	
	$$\left\langle \nabla \varphi, \nabla \left(\frac{|W^+|^2}{R^2}\right)\right\rangle = \frac{2}{R^3}\langle \nabla R,\nabla |W^+|^2\rangle -\frac{4}{R^4}|W^+|^2 |\nabla R|^2.
	$$ Plugging this into (\ref{eq1a3}) yields

	\begin{eqnarray}
		\Delta_{\Psi}\left(\frac{|W^+|^2}{R^2}\right)&=& 4\frac{|Ric|^2}{R^3}|W^+|^2+2\frac{|\nabla R|^2}{R^4}|W^+|^2+\frac{2}{R^2}|\nabla W^+|^2\nonumber\\&& -36\frac{\det W^+}{R^2}-\frac{1}{R^2}\langle (\mathring{Ric}\odot \mathring{Ric})^{+},W^{+}\rangle\\&&-\frac{2}{R^3}\langle \nabla |W^{+}|^{2},\nabla R\rangle,\nonumber
	\end{eqnarray} where $\Psi=f-2\ln R.$ This finishes the proof of the lemma. 
	
\end{proof}

\section{Proof of the Main Results}
\label{rigidityR}

In this section, we will present the proofs of Theorem \ref{thmA}, Theorem \ref{thmB} and Corollary \ref{corC}. In the first part, we adapt the arguments from H.-D. Cao-Ribeiro-Zhou \cite{CRZ}.

\subsection{Proof of Theorem \ref{thmA}}
\begin{proof} By Proposition \ref{propB}, we have
\begin{eqnarray}
\label{eq12sa}
2|\WW^{+}|\Delta_{f}|\WW^{+}| &=& 2|\nabla \WW^{+}|^{2}-2|\nabla |\WW^{+}||^{2}+2|\WW^{\pm}|^{2} -36\, \det \WW^{+}\nonumber\\&& -\langle (\mathring{Ric}\odot \mathring{Ric})^{+},\WW^{+}\rangle,
\end{eqnarray} where we have used that $$\Delta_{f}|\WW^{+}|^{2}=2|\WW^{+}|\Delta_{f}|\WW^{+}|+2|\nabla |\WW^{+}||^{2}.$$ 

Using the Kato inequality (\ref{eqdet}) and Lemma \ref{lemK}, one sees that
\begin{eqnarray}
\label{pl1r}
|\WW^{+}|\Delta_{f}|\WW^{+}| &\geq &  |\WW^{+}|^{2} -\sqrt{6}|\WW^{+}|^{3} -\frac{1}{2}\langle (\mathring{Ric}\odot \mathring{Ric})^{+},\WW^{+}\rangle,\nonumber \\&\geq & |\WW^{+}|^{2} -\sqrt{6}|\WW^{+}|^{3} -\frac{\sqrt{6}}{6}|\mathring{Ric}|^2 |\WW^{+}|.
\end{eqnarray}
Hence, our assumption guarantees that $|\WW^+|\Delta_{f} |\WW^+|$ does not change sign. 

In order to proceed, we need to show that $|\WW^+|$ is $L^2_{f}$-integrable, i.e., $|\WW^+|\in L^2(e^{-f}dV_{g}).$ To do so, we first observe that the assumption also implies that 
\begin{eqnarray*}
\sqrt{6}|\WW^+|^2 &\leq &|\WW^+| -\frac{\sqrt{6}}{6}|\mathring{Ric}|^2\nonumber\\&\leq & \frac{\sqrt{6}}{2}|\WW^+|^2 +\frac{1}{2\sqrt{6}}.
\end{eqnarray*} 
Integrating over $M^4,$ we obtain

\begin{eqnarray*}
\int_{M}|\WW^+|^2 e^{-f} dV_{g}\leq \frac{1}{6}\int_{M}e^{-f} dV_{g}.
\end{eqnarray*} Thus, since $M^4$ has finite weighted volume, i.e., $\int_{M}e^{-f} dV_{g}<\infty$ (see, \cite[Corollary 1.1]{CZ}), one concludes that $|W^+|$ is $L^2_{f}$-integrable. 

Proceeding, we consider a cut-off function $\rho:M\to \Bbb{R}$ such that $\rho=1$ on a geodesic ball $B_{p}(r)$ centered at a fixed point $p\in M$ of radius $r,$ $\rho=0$ outside of $B_{p}(2r)$ and $|\nabla\rho|\leq \frac{c}{r},$ where $c$ is a constant. By our assumption and (\ref{pl1r}), we have
\begin{eqnarray*}
0 &\geq & -\int_{M}\rho^{2} |\WW^+|\Delta_{f}|\WW^{+}| e^{-f} dV_{g}\nonumber\\&=& \int_{M}\left\langle \nabla \left( \rho^2 |\WW^{+}|\right),\nabla |\WW^{+}| \right\rangle e^{-f}dV_{g}\nonumber\\
&=&\int_{M}\left|\rho \nabla |\WW^{+}|+|\WW^{+}|\nabla \rho\right|^{2}e^{-f}dV_{g}-\int_{M} |\WW^{+}|^{2}|\nabla \rho|^{2} e^{-f}dV_{g}. \nonumber
\end{eqnarray*} It follows that
\begin{eqnarray}
\label{okjh}
\int_{M}\left|\nabla \left(\rho |W^{+}|\right)\right|^{2} e^{-f}dV_{g}\leq \int_{M} |W^{+}|^{2} |\nabla \rho|^{2} e^{-f}dV_{g}.
\end{eqnarray} Hence, we deduce that
\begin{eqnarray}
\int_{B(r)}\left|\nabla |W^{+}|\right|^{2} e^{-f}dV_{g}&\leq & \int_{M}\left|\nabla \left(\rho |W^{+}|\right)\right|^{2} e^{-f}dV_{g}\nonumber\\ &\leq &  \int_{M} |W^{+}|^{2} |\nabla \rho|^{2} e^{-f}dV_{g} \nonumber\\ &\leq & \int_{M\backslash B(2r)}|\WW^{+}|^{2}|\nabla \rho |^{2} dV_{g} + \int_{B(2r)\backslash B(r)}|\WW^{+}|^{2}|\nabla \rho |^{2} dV_{g}\nonumber\\&& + \int_{B(r)}|\WW^{+}|^{2}|\nabla \rho |^{2} dV_{g}\nonumber\\
&\leq & \int_{B(2r)\backslash B(r)} |W^{+}|^{2}|\nabla \rho |^{2} e^{-f}dV_{g}\nonumber\\ 
&\leq & \frac{c^2}{r^2}\int_{M} |W^{+}|^{2} e^{-f}dV_{g}.
\end{eqnarray} Since $|W^+|$ is a $L^{2}_{f}$-integrable function on $M^4,$ we conclude that the right hand side tends to zero as $r\to \infty$ and hence, $|W^+|$ must be constant. 

Now, if $|W^+|=0,$ then, by Theorem 1.2 of \cite{CW}, one deduces that the soliton is either Einstein, 
the Gaussian soliton $\Bbb{R}^4,$ the round cylinder $\Bbb{S}^{3}\times\Bbb{R}$, or their quotients. In the case that it is a non-flat Einstein manifold, one invokes Myer's theorem and the Hitchin classification to conclude that it is isometric to either the round sphere $\Bbb{S}^4$ or the complex projective space $\Bbb{CP}^2$ (see \cite[Theorem 13.30]{Besse}).

On the other hand, if $|W^+|$ is a nonzero constant, then it suffices to use (\ref{eq12sa}), (\ref{eqdet}), Lemma \ref{lemK} and the assumption, in order to infer that $\nabla W^+\equiv 0$. Consequently, by Theorem 1.1 in \cite{Wu}, one concludes that $(M^4, g, f)$ is either a finite quotient of $\Bbb{S}^2\times \Bbb{R}^2,$ or compact Einstein with $\nabla W^+\equiv 0$ and $\WW^+$ has 2 distinct eigenvalues. This finishes the proof of Theorem \ref{thmA}. 

\end{proof}

\subsection{Proof of Theorem \ref{thmB}}

\begin{proof}
Initially, one verifies that
\begin{eqnarray*}
\Delta_{\Psi}\left(\frac{|W^+|^2}{R^2}\right)&=& 
2\frac{|W^+|}{R}\Delta_{\Psi}\left(\frac{|W^+|}{R}\right) + 2\frac{|W^+|^2}{R^4}|\nabla R|^2 +\frac{2}{R^2}|\nabla |W^+||^2 \nonumber\\&&-\frac{2}{R^3}\langle \nabla R,\nabla |W^+|^2\rangle.
\end{eqnarray*} Now, substituting this data into Lemma \ref{lemB}, one obtains that 
\begin{eqnarray*}
2\frac{|W^+|}{R}\Delta_{\Psi}\left(\frac{|W^+|}{R}\right)&=& 4\frac{|Ric|^2}{R^3}|W^+|^2+\frac{2}{R^2}\left(|\nabla W^+|^2 - |\nabla |W^+||^2\right)\nonumber\\&& -36\frac{\det W^+}{R^2}-\frac{1}{R^2}\langle (\mathring{Ric}\odot \mathring{Ric})^{+},W^{+}\rangle.\nonumber
\end{eqnarray*} 
Thereby, by the Kato inequality, (\ref{eqdet}) and Lemma \ref{lemK}, we compute
\begin{eqnarray}
\label{wlkj}
2\frac{|W^+|}{R}\Delta_{\Psi}\left(\frac{|W^+|}{R}\right)&\geq& 4\frac{|Ric|^2}{R^3}|W^+|^2 - \frac{2\sqrt{6}}{R^2}|W^+|^3\nonumber\\&& -\frac{1}{R^2}\langle (\mathring{Ric}\odot \mathring{Ric})^{+},W^{+}\rangle \nonumber\\
 & \geq & 4\frac{|\mathring{Ric}|^2}{R^3}|W^+|^2 +\frac{|W^+|^2}{R}- \frac{2\sqrt{6}}{R^2}|W^+|^3\nonumber\\ &&
-\frac{1}{R^2}\frac{\sqrt{6}}{3}|W^+||\mathring{Ric}|^2\nonumber\\
\label{wlkj2} &=&\frac{1}{R}\left(1-2\sqrt{6}\frac{|W^+|}{R}\right)\left(|W^+|^2-\frac{\sqrt{6}}{3}\frac{|\mathring{Ric}|^2}{R}|W^+|\right).
\end{eqnarray} Hence, our assumption guarantees that $\frac{|W^+|}{R}\Delta_{\Psi}\left(\frac{|W^+|}{R}\right)$ does not change sign. 

Now, we need to show that $\frac{|W^+|}{R}$ is a $L^{2}_{\Psi}$-integrable function on $M^4.$ Indeed, we observe that 
\begin{equation}
\label{erf}
\int_{M}\frac{|W^+|^2}{R^2}\,e^{-\Psi}dV_{g}=\int_{M}\frac{|W^+|^2}{R^2}\,e^{-f}e^{2\ln R}dV_{g}=\int_{M}|W^{+}|^2e^{-f}dV_{g}.
\end{equation} At the same time, our assumption gives 
\begin{equation}
\label{erf2}
\int_{M}|W^+|^2 e^{-f} dV_{g}\leq \frac{1}{24} \int_{M}R^2 e^{-f} dV_{g}\leq \frac{1}{6}\int_{M}|Ric|^2 e^{-f} dV_{g}.
\end{equation} By Theorem 1.1 in \cite{MS}, we have $$\int_{M}|Ric|^2 e^{-f} dV_{g}<\infty.$$ Thus, it follows from (\ref{erf}) and (\ref{erf2}) that $\frac{|W^+|}{R}$ is a $L^{2}_{\Psi}$-integrable.

Next, we are going to apply a cut-off function argument similar as in the proof of Theorem \ref{thmA} to conclude that $\frac{|W^+|}{R}$ is constant. Indeed, we set a cut-off function $\rho:M\to \Bbb{R}$ such that $\rho=1$ on a geodesic ball $B_{p}(r)$ centered at a fixed point $p\in M$ of radius $r,$ $\rho=0$ outside of $B_{p}(2r)$ and $|\nabla\rho|\leq \frac{c}{r},$ where $c$ is a constant. Hence, taking into account that $\frac{|W^+|}{R}$ is a $\Psi$-subharmonic function on $M^4$ and $R>0$ unless $(M^4,\,g,\,f)$ is the Gaussian shrinking soliton, one easily verifies that

\begin{eqnarray*}
0 &\geq & -\int_{M}\rho^{2}\frac{|W^+|}{R}\Delta_{\Psi}\left(\frac{|W^{+}|}{R}\right)e^{-\Psi} dV_{g}\nonumber\\ &=&\int_{M}\left|\rho \nabla \left(\frac{|W^{+}|}{R}\right)+\frac{|W^{+}|}{R}\nabla \rho\right|^{2}e^{-\Psi}dV_{g}-\int_{M}\frac{|W^{+}|^{2}}{R^{2}}|\nabla \rho|^{2} e^{-\Psi}dV_{g},\nonumber
\end{eqnarray*} so that

\begin{eqnarray}
\int_{M}\left|\nabla \left(\rho \frac{|W^{+}|}{R}\right)\right|^{2} e^{-\Psi}dV_{g}\leq \int_{M} \frac{|W^{+}|^{2}}{R^{2}}|\nabla \rho|^{2} e^{-\Psi}dV_{g}.
\end{eqnarray} It follows that

\begin{eqnarray}
\int_{B(r)}\left|\nabla \left( \frac{|\WW^{+}|}{R}\right)\right|^{2} e^{-\Psi}dV_{g}&\leq & \int_{M}\left|\nabla \left(\rho \frac{|\WW^{+}|}{R}\right)\right|^{2} e^{-\Psi}dV_{g}\nonumber\\ &\leq & \int_{B(2r)\backslash B(r)} \frac{|\WW^{+}|^{2}}{R^{2}}|\nabla \rho|^{2} e^{-\Psi}dV_{g}\nonumber\\ &\leq & \frac{c^2}{r^2}\int_{M} \frac{|W^{+}|^{2}}{R^{2}} e^{-\Psi}dV_{g}.
\end{eqnarray} Since $\frac{|\WW^+|}{R}$ is a $L^{2}_{\Psi}$-integrable function on $M^4,$ we conclude that the right hand side tends to zero as $r\to \infty$ and consequently, $\frac{|W^+|}{R}$ must be constant on $M^4,$ as asserted.

If $\frac{|\WW^+|}{R}=0,$ then $\WW^+\equiv 0$. As in the proof of Theorem \ref{thmA}, the soliton must be isometric to either $\Bbb{S}^4$,  $\Bbb{CP}^2,$ or a finite quotient of the round cylinder $\Bbb{S}^{3}\times\Bbb{R}.$

Otherwise, if $\frac{|W^+|}{R}$ is a nonzero constant, each inequality above becomes an equality. In particular, by Lemma \ref{Katoinequality}, one obtains that $\nabla \WW^+=\nu \otimes \WW^+$ for some one-form $\nu.$ Besides, it follows from Lemma \ref{lemK} that $\WW^+$ has $\lambda_2=\lambda_3$ at each point. Finally, it suffices to use Lemma \ref{kahlerform} to conclude that $(M^4,\, g)$ is a locally K\"ahler-Ricci soliton.  

\end{proof}

\subsection{Proof of Corollary \ref{corC}}

\begin{proof}
To begin with, we follow the initial steps in the proof of Theorem \ref{thmB} up till (\ref{wlkj2}) in order to obtain

\begin{equation}
\label{rgt5}
2\frac{|W^+|}{R}\Delta_{\Psi}\left(\frac{|\WW^+|}{R}\right)  \geq  
\frac{1}{R}\left(1-2\sqrt{6}\frac{|\WW^+|}{R}\right)\left(|\WW^+|^2-\frac{\sqrt{6}}{3}\frac{|\mathring{Ric}|^2}{R}|\WW^+|\right).
\end{equation} Thereby, our assumption guarantees that the right hand side of (\ref{rgt5}) is nonnegative and therefore, by using the Maximum Principle, we conclude that $\frac{|\WW^+|}{R}$ is a constant. In particular, each previous inequality obtained in the establishment of (\ref{rgt5}) becomes an equality. Now, observe that $\WW^+=0$ leads to a contradiction with the assumption (\ref{cond3}). Consequently, $\WW^+\neq 0$ and we may use again the equality case in Lemma \ref{Katoinequality}, Lemma \ref{lemK} and Lemma \ref{kahlerform} in order to infer that $(M^4,\, g)$ is a locally K\"ahler-Ricci soliton, which gives the desired result. 
\end{proof}

\end{document}